\documentclass[a4paper,11pt,twoside,reqno]{amsart}  
\DeclareSymbolFontAlphabet{\mathcal}{symbols}

\usepackage{lscape}
 \usepackage[ansinew]{inputenc}
 \usepackage[dvips]{graphicx}
\usepackage[psamsfonts]{amssymb}
 \usepackage[all]{xy} 
 \usepackage{color}

\usepackage[psamsfonts]{eucal}
\usepackage{amssymb, amsmath}
\usepackage{amsthm}
\usepackage{geometry}
\usepackage{enumerate}
\usepackage{float}

\theoremstyle{plain}

\newtheorem{theoremv}{Theorem}

\newtheorem{proposition}{Proposition} 

\newtheorem{lemma}[proposition]{Lemma}
\newtheorem{corollary}[proposition]{Corollary}

\theoremstyle{definition}
\newtheorem{definition}[proposition]{Definition}

\theoremstyle{remark}

\newcommand{\secref}[1]{Section~\ref{#1}}

\newcommand{\propref}[1]{Proposition~\ref{#1}}
\newcommand{\lemref}[1]{Lemma~\ref{#1}}

\def\cA{{\mathcal A}}
\def\cB{{\mathcal B}}

\def\cG{{\mathcal G}}

\def\cK{{\mathcal K}}
\def\cL{{\mathcal L}}

\def\cP{{\mathcal P}}
\def\cT{{\mathcal T}}

\def\L{\mathbb{L}}

\def\Q{\mathbb{Q}}
\def\R{\mathbb{R}}

\def\hL{{\widehat{\mathbb L}}}

\def\ad{{\rm ad}}

 \newcommand{\dgl }{\text{\rm DGL}}
  \newcommand{\cdgl}{\text{\rm cDGL}}

 \newcommand{\MC}{\operatorname{{\rm MC}}}

\title[Maurer-Cartan elements in Lie models]{Maurer-Cartan elements in the {L}ie models of finite simplicial complexes}

\author{Urtzi Buijs}
\address{Departamento de Algebra, Geometr\'ia y Topolog\'ia\\
         Universidad de M\'alaga\\
        Ap. 59\\
         29080-M\'alaga,\\
         Espa\~na}
\email{ubuijs@uma.es}

\author{Yves F\'elix}
\address{Institut de Math\'ematiques et Physique\\
         Universit\'e Catholique de Louvain-la-Neuve\\
         Louvain-la-Neuve\\
         Belgique}
\email{Yves.felix@uclouvain.be}

\author{Aniceto Murillo}
\address{Departamento de Algebra, Geometr\'ia y Topolog\'ia\\
         Universidad de M\'alaga\\
        Ap. 59\\
         29080-M\'alaga,\\
         Espa\~na}
         \email{aniceto@uma.es}

\author{Daniel Tanr\'e}
\address{D\'epartement de Math\'ematiques, UMR 8524\\
         Universit\'e de Lille~1\\
         59655 Villeneuve d'Ascq Cedex\\
         France}
\email{Daniel.Tanre@univ-lille1.fr}

\thanks{The first author has been partially supported by the Ram\'on y Cajal MINECO programme. 
The first and third authors have been partially supported by the Junta de Andaluc\'\i a grant FQM-213. 
The fourth author has been partially supported by the  ANR-11-LABX-0007-01  ``CEMPI''.
The authors are partially supported by the MINECO grant MTM2013-{41768-P}}

\begin{document}
  
\date{\today}

\begin{abstract}
In a previous work, we have associated a complete differential graded Lie algebra 
to any finite simplicial complex in a functorial way. 
Similarly, we have also a realization functor from the category of complete differential graded Lie algebras
to the category of simplicial sets.
We have already interpreted the homology of a Lie algebra
in terms of homotopy groups of its realization.
In this paper, we begin a dictionary between models 
and simplicial complexes by establishing a correspondence 
between the Deligne groupoid of the model and the connected components of the finite simplicial complex.
\end{abstract}

\maketitle

Let $\MC(L)$ be the set of Maurer-Cartan elements of 
a differential graded Lie algebra $(L,d)$ over $\Q$ (henceforth $\dgl$).
The group $L_{0}$ of elements of degree~0, endowed with the Baker-Campbell-Hausdorff product, acts on $\MC(L)$ by
$$x\cG z= e^{\ad_{x}}(z)-\frac{e^{\ad_{x}}-1}{\ad_{x}}(d x),
$$
with $x\in L_{0}$ and $z\in\MC(L)$.
We denote by $\widetilde{\MC}(L)$ the orbit space for this action.

\smallskip
In \cite{BFMT1}, we construct a functor $\cL$ from the category of finite simplicial complexes to the category of
complete differential graded Lie algebras (henceforth $\cdgl$), $X\mapsto \cL_{X}$. 
Rational homotopy has been mainly  introduced and used for simply connected spaces
(\cite{FHT1}, \cite{Qu}, \cite{Su}). 
In \cite{Su}, there is also an extension to  non-simply connected spaces over $\R$ via fiber bundles (see
\cite{GHT} for an adaptation to $\Q$). Recently, the classical approach has been extended
to non-simply connected spaces in  \cite{FHT2} and the
functor ${\cL}$ gives the corresponding extension for DGL's.

\medskip
In this paper we prove the following relation between ${\cL}_X$ and the topology of $X$.

\begin{theoremv}
For any finite simplicial complex $X$ there is a bijection
$$\pi_0(X_+) \cong \widetilde{\MC} ({\cL}_X),$$
where $X_{+}=X\sqcup \{\ast\}$.
\end{theoremv}

The case of the interval $X=[0,1]$ was solved in \cite{BFMT2}.
In \secref{sec:MC}, we make the necessary recalls on Maurer-Cartan elements and the functor $\cL$. \secref{sec:LX}
is devoted to a decomposition  of $\cL_{X}$ when $X$ is connected. Finally, the proof of
the Theorem is done in \secref{sec:MCLX}.

\section{Functor $\cL$ and Maurer-Cartan elements}\label{sec:MC}

Recall that a \dgl~$(L,d)$ is \emph{complete} if $L=\varprojlim_n L/L^{[n]}$ where $L^{[n]}$ denotes the sequence of ideals defined by
$$L^{[1]}= L\,, \hspace{5mm}\text{and } L^{[n+1]} = [L, L^{[n]}]\,, \hspace{2mm} n\geq 2\,.$$ 
When $V$ is finite dimensional,  
$\widehat{\L}(V)=\lim_{n}\,\L(V)/\L(V)^{[n]}$ 
is the completion of the free graded Lie algebra $\L(V)$. 

\medskip
Let $(L,d)$ be a \cdgl. An element $u\in L_{-1}$ is a \emph{Maurer-Cartan element} if
$$du = -\frac{1}{2} [u,u]\,.$$
In \cite{LS}, R. Lawrence and D. Sullivan construct a \cdgl~${\cL}_I$ that is, in a sense that we will precise later, a model for the interval $I=[0,1]$. More precisely,
$${\cL}_I= (\widehat{\mathbb L}(a,b,x), d)\,,$$
where $a$ and $b$ are Maurer-Cartan elements and $x$ is an element of degree $0$ with
$$dx= \ad_xb+ \frac{\ad_x}{e^{\ad_x}-1}(b-a) = [x,b] + \sum_{n=0}^\infty \frac{B_n}{n!} \ad_x^n(b-a)\,.$$
Here the $B_n$ are the well known Bernoulli numbers. This model has been described in detail in \cite{PaTa}, \cite{BM}. 

In a \cdgl~$(L,d)$, two Maurer-Cartan elements $u_1$ and $u_2$ are \emph{equivalent} if they are in the same orbit for the gauge action. By construction, this is equivalent to the existence of  a morphism of \dgl's,
$$f \colon {\cL}_I \to (L,d)$$
with $f(a) = u_1$ and $f(b)= u_2$. The map $f $ is  called \emph{a path from $u_1$ to $u_2$.} The set of equivalence classes of Maurer-Cartan elements is denoted $\widetilde{MC}(L)$. 

Our purpose is the determination of $\widetilde{MC}(L)$ for a family of \cdgl's directly related to topology. 
In fact the \cdgl~${\cL}_I$ is the first example of a Lie model for a general simplicial complex. 
More generally,  there is a functor ${\cL}$, unique up to isomorphism, $X\mapsto {\cL}_X$,  
from the category of finite simplicial complexes to the category of \cdgl's. As any finite simplicial complex is
a subcomplex of some $\Delta^n$, it is sufficient to construct the models, ${\cL}_{\Delta^n}$, of the $\Delta^n$'s.

\begin{proposition} \label{unicity}{\cite[Theorem 2.8]{BFMT1}}  
 The \cdgl~${\cL}_{\Delta^n}$ is defined, up to isomorphism, by the following properties.
 \begin{enumerate}[(i)]
 \item  The \cdgl's ${\cL}_{\Delta^n}$ are natural with respect to the injections of the subcomplexes $\Delta^p$, for all $p<n$.
 \item For $n=0$, we have ${\cL}_{\Delta^0}= (\hL(a), d)$ where $a$ is a Maurer-Cartan element.
 \item The linear part $d_1$ of the differential of $\cL_{\Delta^n}$ is the desuspension of the differential $\delta$ of the chain complex $C_*(\Delta^n)$.
 \end{enumerate}
 \end{proposition}
 In the case  $\Delta^1=[0,1]$, we recover the Lawrence-Sullivan construction. 
 For each finite simplicial complex, $X$, contained in $\Delta^n$, the Lie subalgebra $\widehat\L(s^{-1}C_*(X))$ is preserved by the differential of $\cL_{\Delta^n}$ and gives a \emph{model} $\cL_X$ of $X$.

When $a$ is a Maurer-Cartan element in ${\cL}_X$, we denote by $d_a$ the perturbed differential $d_a= d +  \ad_a$. The first properties of 
${\cL}_X=(\widehat{\L}(W),d)$ are contained in the following statements  extracted from \cite{BFMT1} and \cite{BFMT3}.
\begin{enumerate}[(i)]
\item If $d_1$ denotes the  linear part   of the differential $d$, then $(W, d_1)$ is isomorphic  to the desuspension 
of the simplicial chain complex $C_*(X)$ of $X$.
\item If $f \colon X\to Y$ is the inclusion of a subcomplex, then ${\cL}_f \colon {\cL}_X\to {\cL}_Y$ is equal to $\widehat{\L}({s^{-1}C_*(f)})$.
\item $H({\cL}_X) = 0$ (\cite{BFMT3}, Theorem 4.1).
\item If $X$ is simply connected, and $a$ is the Maurer-Cartan element associated to a $0$-simplex, then
 $(\widehat{\L}(W),d_a)$ is quasi-isomorphic to the usual rational Quillen model of $X$   
 (\cite{BFMT1}, Theorem 7.4(ii)).
\item If $X$ is connected and $a$ is the Maurer-Cartan element associated to a $0$-simplex, then 
$H_0(\widehat{\L}(W),d_a)$ is isomorphic to the Malcev Completion of $\pi_1(X)$ (\cite{BFMT1}, Theorem 9.1).
\end{enumerate}

 Recall  that the Lawrence-Sullivan interval ${\cL}_I$ is isomorphic to the cylinder construction
 (\cite{Ta}) on a Maurer-Cartan element (\cite[Theorem 6.3]{BFMT3}). More precisely,  
 consider the \cdgl~$(\widehat{\L}(a,c,y),d)$ with $\vert y\vert = 0$, $\vert c\vert = -1$, 
 $da= -\frac{1}{2}[a,a]$, $dy= c $ and $dc= 0$ 
 that we equip  with a derivation $s$ of degree $+1$,  defined by $s(a)= y$, $s(c)=s(y)= 0$. 
 Then the morphism
\begin{equation}\label{equa:isoLS}
\psi \colon (\widehat{\L}(a,b,x),d) \to (\widehat{\L}(a,c,y),d)
\end{equation}
defined by $\psi (a)= a$, $\psi (b) = e^{sd+ds}(a)$, $\psi (x)= y$ is an isomorphism of \dgl's.
In particular,
$$\psi (b) = a+c+ \sum_{n\geq 1} \frac{(sd)^n}{n!} (a)= e^{\ad_{-y}}(a) + \frac{e^{\ad_{-y}}-1}{\ad_{-y}} (c)\,.$$

\begin{definition}\label{def:MCr}
 Two Maurer-Cartan elements $u,v$ in a \cdgl$~(\widehat{\L}(V),d)$ are called \emph{equivalent of order $r$} if there is a morphism 
$$\varphi \colon (\widehat{\L}(a,b,x),d) \to (\widehat{\L}(V),d)$$
with $\varphi (x)\in \L^{\geq r}(V)$, $\varphi (a) = u$ and $\varphi (b)= v$. We denote this relation by
$u\sim_{O(r)} v$.
\end{definition}

This relation is a key-point in the proof of \propref{prop:cscmodel}. 
We end this section with two properties of $\sim_{O(r)}$.

\begin{lemma}\label{lem:MCr}
  Let   $u$ be a Maurer-Cartan element in $(\widehat{\L}(V),d) $. We suppose  $u = v+ w$ 
  with $w\in \L^{\geq r} (V)$, and the existence of  an element $z\in \L^{\geq r}(V)$ with $dz=w+t$ and
   $t\in \L^{\geq r+1}(V)$.
   Then, we have  $u\sim_{O(r)} v+w' $ with $w'\in \mathbb L^{\geq r+1}(V)$.
\end{lemma}

\begin{proof}
Let $f \colon (\widehat{\L}(a,c,y),d)\to (\widehat{\L}(V),d)$ be the morphism defined by $f(a)= u$, $f(y)= -z$ 
and $f(c) =  -dz$. Then $f\circ \psi$ is a path in $(\widehat{\L}(V),d)$ 
with $f\psi (a)= u$, $f\psi (x)= -z$. 
To determine $f\psi(b)$, we first observe that
$$\psi (b) = a+c+ \sum_{n\geq 1} \frac{(sd)^n}{n!} (a)\,.$$ 
Remark also that $f(sd)^n(a)\in \L^{\geq r+1}(V)$, for $n\geq 1$.
Therefore 
$$f\circ \psi (b) \in f(a) + f(c) + \L^{\geq r+1}(V)
= u-dz+ \L^{\geq r+1}(V)=v-t+ \L^{\geq r+1}(V),$$
with $t\in \L^{\geq r+1}(V)$.
\end{proof}
 
\begin{lemma}\label{lem:seqMC}
   Let $(u_r)_{r\geq n_0}$
   be a sequence of Maurer-Cartan elements in $(\widehat{\L}(V),d)$ such that $u_r = z+ v_r $ with   $v_r\in \L^{\geq r}(V)$. 
   If $u_r\sim_{O(r)} u_{r+1}$ for each $r\geq n_0$,   then  we have $u_{n_0}\sim_{O(n_0)} z $.
\end{lemma}

\begin{proof} By hypothesis, for $r\geq n_0$ there is a morphism
$$\varphi_r \colon (\widehat{\L}(a,b,x),d)\to (\widehat{\L}(V),d)
$$
with $\varphi_r(a) = u_r$, $\varphi_r(b) = u_{r+1}$ and $\varphi_r(x)\in \mathbb L^{\geq r}(V)$. 
For $r>n_0$, we define $w_r$ to be the Baker-Campbell-Hausdorff product 
$$w_r = \varphi_{n_0}(x)* \varphi_{n_0+1}(x)*\dots * \varphi_{r-1}(x)\,.$$
 From the associativity established in \cite{LS}, the element $w_r$ is a path from $u_{n_0}$ to $u_r$. 
We form the infinite product 
$$w= \varphi_{n_0}(x)* \varphi_{n_0+1}(x)*\dots $$
which is   well defined in $\widehat{\L}(V)$ as the limit of the $w_r$.  
Now we claim that the element $w$ is a path of order $n_0$ from $u_{n_0}$ to   $z$; i.e., we have
$u_{n_0}\sim_{O(n_0)} z$.
Consider   the element $$y = dw - [w,z] - \sum_{n\geq 0} \frac{B_n}{n!} \ad_w^{n} (z- u_{n_0})\,,$$ 
where the $B_n$ are the Bernoulli numbers.
The element $y$ has the same image in $\L(V)/ \L^{\geq {r}}(V)$   than  
$$dw_r -[w_r, u_r] - \sum_{n\geq 0} \frac{B_n}{n!} \ad_{w_r}^n (u_r-u_{n_0}).$$
This last expression is equal to $0$ because $w_r$ is a path from $u_{n_0}$ to $u_r$. This implies  $y= 0$ and  proves the result.
\end{proof}

\section{Model of a finite connected simplicial complex}\label{sec:LX}

\begin{proposition}\label{prop:cscmodel}

 Let $X$ be a  connected finite simplicial complex of dimension $n$, then we have an isomorphism
 of \cdgl's
$${\cL}_X \cong (\widehat{\L}(V),d) \,\widehat{\amalg}_{i} \,(\widehat{\L}(u_i, v_i),d)$$
where $dv_i = u_i$,  $du_{i}=0$, $V=V_{\leq n-1}$, $V= \Q a \oplus V_{\geq 0}$, $a$ is a Maurer-Cartan element
and $\widehat\amalg$ denotes the completion of the coproduct. Moreover,
the differential of any $x\in V_{\geq 0}$ verifies
$$dx+[a,x]\in \hL^{\geq 2}(V_{\geq 0}).$$ 
\end{proposition}

\begin{proof} By \lemref{lem:onedc}, this is true if $\dim X= 1$. 
Proceed by induction on $n$.
We can therefore suppose that 
$X = Y\cup\cup_{j=1}^k \Delta^n_{j}$ and 
$$({\cL}_Y,d) \cong (\widehat{\L}(V),d) \,\widehat{\amalg}_{i} \,(\widehat{\L}(u_i, v_i),d)$$
with $n\geq 2$, $\dim Y\leq n-1$,  $V=V_{\leq n-2}=\Q a\oplus W$, $W=W_{\geq 0}$, $|v_{i}|\leq n-2$, $dv_{i}=u_{i}$.
We set $u'_i = u_i + [a,v_i]$ and we get an isomorphism of \dgl's
$$(\widehat{\L}(V), d_a) \,\widehat{\amalg}\,\widehat{\amalg}_{i}\,(\widehat{\L}(u_i', v_i),d_a)\to ({\cL}_Y, d_a)\,,$$
with $d_av_i = u_i'$, $d_{a}u_{i}'=0$. 
Now, by construction of the model  $\cL_{X}$, there are cycles $\Omega_{j}\in ({\cL}_Y)_{n-2}$ such that
$$({\cL}_X, d_a) = ({\cL}_Y \,\widehat{\amalg}\,\widehat{\amalg}_{j=1}^k \,\L(x_{j}),d_a)\,,
 \hspace{1cm} \vert x_{j}\vert = n-1\,, \hspace{5mm} 
d_ax_{j}= \Omega_{j}\,.$$
Since the inclusion 
$(\widehat{\L}(V), d_a) \hookrightarrow (\widehat{\L}(V),d_a) \,\widehat{\amalg}\,\widehat{\amalg}_{i}\,
(\widehat{\mathbb L}(u_i',v_i),d_a)$ 
is a quasi-isomorphism, we can choose 
$\Omega_{j}\in \widehat{\L}(W)$.  

Let $(x_{j})_{j\in \cA}$ the family of the $x_{j}$'s such that the differential $dx_{j}=\Omega_{j}$
has a non-zero linear part $\Omega^1_{j}$.
We set $\cB=\{1,\ldots,k\}\backslash \cA$ and denote by $\cK$ the ideal generated by
$\{x_{j},\Omega_{j}^1\mid j\in\cA\}$.
If $V'$ is a direct summand of $\oplus_{j\in\cA}\Q \Omega^1_{j}$ in $V$,
we have an isomorphism $(\hL(V'),d)\cong (\hL(V),d)/\cK$.
From \cite[Proposition 2.4]{BFMT1}, we deduce that the canonical surjection
$\rho\colon (\hL(V),d)\to (\hL(V),d)/\cK$
is a quasi-isomorphism. 
Since the \dgl~$(\hL(V'),d)$ is cofibrant (\cite[Proposition 5.4]{BFMT3}), we may lift $\rho$ in a quasi-isomorphism
$$\varphi\colon (\hL(V'),d)\,\widehat{\amalg}\,\widehat{\amalg}_{j\in\cA}\hL(x_{j},\Omega_{j})\to (\hL(V),d)$$
and get an isomorphism
$$\cL_{X}\cong \hL(V'\oplus\,\oplus_{j\in\cB}\Q x_{j})\,\widehat{\amalg}\,(\widehat{\amalg}_{j\in\cA}\hL(x_{j},\Omega_{j})
\,\widehat{\amalg}_{i}\hL(u_{i},v_{i})).$$
\end{proof}

\begin{lemma}\label{lem:onedc}
 Let $X$ be a 1-dimensional  connected finite simplicial complex, then we have an isomorphism
 of \cdgl's
$${\mathcal L}_X \cong (\widehat{\L}(V),d)\, \widehat{\amalg}\, (\widehat{\L}(u_i,  v_i), dv_i = u_i)\,,$$
with $V = \Q a \oplus V_0$, $da=-\frac{1}{2}[a,a]$ and $dx= -[a,x]$ for any $x\in V_0$.
\end{lemma}

\begin{proof}
 Let $x_0$ be a vertex of $X$ and let $a$ denote the corresponding Maurer-Cartan element in ${\mathcal L}_X$. 
 By hypothesis $X$ is a connected finite graph, and we denote by $\cT$   a maximal tree in $X$. 
 For each vertex $v_i$ different from $x_0$, there is a unique path $\cP_{v_i}\in\cT$ of minimal length from 
 $x_0$ to $v_i$. 
 We   remark that each edge in $\cT$ is the terminal edge of some path $\cP_{v_i}$ for some vertex 
 $v_i$ different from $x_0$.
 The vertices $v_i$ correspond to Maurer-Cartan elements $a_i$ in ${\cL}_X$. 
 To each path $\cP_{v_i}$ we associate the Baker-Campbell-Hausdorff product $p_i$ of the edges  composing 
 this path.  

If $b_k$ is an edge which does not belong to $\cT$,   we denote by  $v_{k_0}$ and $v_{k_1}$ its endpoints. 
If each of them is different from $x_0$,  we form the loop consisting of the path 
$\cP_{v_{k_0}}$ followed by $b_k$ and $(\cP_{v_{k_1}})^{-1}$. 
If $v_{k_0}= x_0$, we form the loop consisting of $b_k$ and $(\cP_{v_{k_1}})^{-1}$
and do similarly if $v_{k_1}= x_0$.
We denote then by $c_k$ the Baker-Campbell-Hausdorff product of the edges composing this loop.

From these two constructions, we get a morphism of \dgl's
$$f \colon ({\cL}',d) := (\widehat{\L}(a,a_i, p_i, c_k),d) \to {\mathcal L}_X.$$
The map $f$ induces an isomorphism on the indecomposable elements and thus it is an isomorphism.
In $({\cL}',d)$, for each $i$, $(\widehat{\L}(a,a_i, p_i),d)$ is a Lawrence-Sullivan interval connecting $a$ to $a_i$. 
On the other hand (see \cite[Proposition 2.7]{BFMT1}), for each $k$ we have $dc_k = -[a,c_k]$. 

Recall now from (\ref{equa:isoLS}) that for each $i$, there is an isomorphism
$$\psi_i \colon (\widehat{\L}(a,a_i, p_i),d)\to (\widehat{\L}(a, u_i, v_i),d)$$
with $\psi (a) = a$, $\psi(p_{i})=v_{i}$, $du_i=0$ and $dv_i= u_i$. 
The morphisms $\psi_i$ can be pasted together and give an isomorphism
$$\psi \colon ({\cL}',d) \to (\widehat{\L}(a, u_i, v_i, c_k),d)$$
with $dc_k = -[a,c_k]$ and $dv_i = u_i$.
Therefore
$${\cL}_X \cong (\widehat{\L}(V),d)\,\widehat{\amalg} \,(\widehat{\L}(u_i, v_i),d)$$
with $V = \Q a \oplus V_0$ and $dx= -[a,x] $ for any $x\in V_0$.  
\end{proof}

\begin{corollary}\label{cor:a}
 With the notations of \propref{prop:cscmodel}, we have
$$\widetilde{MC}({\cL}_X) = \widetilde{MC}(\widehat{\L}(  V),d)\,.$$
\end{corollary}

\begin{proof} This follows directly from \cite[Proposition 2.4]{BFMT3}.
\end{proof}
\section{Maurer Cartan elements and connected components}\label{sec:MCLX}

\begin{proof}[Proof of the Theorem.] 
Let $X$ be a finite simplicial complex and denote by $X_i$ its connected components
for $i= 1, \dots , k$. Then
$$ {\cL}_X= \widehat{\amalg}_{i= 1}^k {\cL}_{X_i}\,.$$
For each $i=1,\ldots,k$ we have
 $${\cL}_{X_i}\cong (\widehat{\L}(  V(i),d)\widehat{\amalg} (\widehat{\L}(u_{ij}, v_{ij}),d),$$ 
 with 
$d(u_{ij})= v_{ij}$ and $V(i) = \Q a_i \oplus V(i)_{\geq 0}$ verifies the properties established in \propref{prop:cscmodel}.
Moreover, we deduce from Corollary~\ref{cor:a} 
$$\widetilde{MC}({\cL}_X) = \widetilde{MC}(\,\widehat{\amalg}_{i=1}^k (\widehat{\L}(  V(i)),d)\,)\,.$$
A Maurer-Cartan element $u\in \cL_{X}$ can be written in the form
$$u = \sum_{i=1}^k \lambda_i a_i + \mu\,,
$$
where $\mu$ is a decomposable element and $\lambda_i \in \Q$.
 From a short computation, we observe that all the numbers $\lambda_{i}$, except at most one,
 are equal to zero. 
 
 \smallskip
$\bullet$ {If $\lambda_1 \neq 0$, then $\lambda_1 = 1$} and we set $a=a_{1}$, $V=V(1)$ and $W=\oplus_{i\geq 2}V(i)$.
 We denote by  $E_r$ the  subvector space of $\cL_{X}$ 
 generated by the Lie words containing exactly $r$ elements of $V_{\geq 0}$.
The differential $d$ can be written as a series 
$d= \sum_{i\geq 1} d_i$, with $d_i(V) \subset E_i$. 
By hypothesis, we have $d_1(v) = -[a,v]$ if $v\in V_{\geq 0}$ and $d_{1}(w)=0$ if $w\in W$.
Remark now that since $a$ is in degree $-1$ and $V\oplus W$ is finite dimensional,  
the ideal $E_{\geq 1}$ generated by $V_{\geq 0}$ is the free complete \dgl~on the elements 
$a^r\boxtimes v_k := \ad_a^r(v_k)$ and 
$a^r\boxtimes w_k := \ad_a^r(w_k)$, where $r\geq 0$,  
the $v_k$'s run over a graded basis of $V_{\geq 0}$ and the $w_{k}$ over a graded basis of $W$. 
Recall $v\in V_{\geq 0}$ and $w\in W$.
A simple computation gives
$$d_1(a^r\boxtimes v) = 
\left\{\begin{array}{cl} -a^{r+1}\boxtimes v\,, \hspace{5mm}\mbox{} 
& \mbox{if $r$ is even,}\\
0\,,
&\mbox{if $r$ is odd,}
\end{array}
\right.$$
$$d_1(a^r\boxtimes w) = 
\left\{\begin{array}{cl}
0\,,
&\mbox{if $r$ is even,}\\
 -a^{r+1}\boxtimes w\,, \hspace{5mm}\mbox{} 
& \mbox{if $r$ is odd.}
\end{array}
\right.$$
The derivation defined by $\theta = -\ad_a-d_1$ verifies
$$\theta (a^r\boxtimes v) = \left\{ \begin{array}{cl} 0\,, \hspace{1cm} \mbox{} 
& \mbox{if $r$ is even,}\\
-a^{r+1}\boxtimes v\,, 
& \mbox{if $r$ is odd,}
\end{array}\right.$$
$$\theta (a^r\boxtimes w) = \left\{ \begin{array}{cl} 
-a^{r+1}\boxtimes w\,, \hspace{1cm} \mbox{} 
& \mbox{if $r$ is even,}\\
0\,, 
& \mbox{if $r$ is odd.}
\end{array}\right.$$
Clearly we have $\theta^2 = 0$ and $H(E_{\geq 1}, \theta) = \widehat{\L}(V)$. In particular
$$H_{-1}(E_{\geq 1}, \theta) = 0\,.$$
We construct a sequence of Maurer-Cartan elements $(u_n)$ such that 
$u_{1} = u$, $u_n-a\in E_{\geq n}$ and $u_n\sim_{O(n)} u_{n+1}$. 
Suppose $u_n$ has been constructed, then we can write it as
$$u_n = a+ \omega_n + \gamma\,, \quad 
 \text{with}\;\;\omega_n \in E_n\,, \hspace{2mm}\gamma \in E_{>n}\,.$$
Since $u_n$ is a Maurer-Cartan element, we have $d_1(\omega_n) = -[a,\omega_n]$ and $\theta (\omega_n)= 0$.
From  $H_{-1}(E_{\geq 1}, \theta) = 0$, we deduce the existence of $t\in E_n$ such that $\omega_n = \theta (t)$.
This implies 
$$\omega_n = -[a,t] -d_1(t)\,.$$
Recall from (\ref{equa:isoLS}) the morphism
$$\psi \colon (\widehat{\L}(a,b,x),d)\to (\widehat{\L}(a,e,c),d)$$ 
and  construct a morphism 
$\mu \colon (\widehat{\L}(a,e,c),d)\to (\widehat{\L}(\Q a\oplus V),d)$,
 by $\mu (a) = u_n$, $\mu(e) = t$ and $\mu (c) = dt$. 
 A short computation gives
$$\mu\circ \psi (b) = a+\gamma '\,, \hspace{1cm} \gamma'\in E_{>n}\,.$$ 
The path $\mu\circ \psi$ defines $u_{n+1}$ such  that $u_n\sim_{O(n)} u_{n+1}$
and the result follows  from \lemref{lem:seqMC}.

\smallskip
$\bullet$ {Suppose now $\lambda_{i}=0$ for $i=1,\ldots,k$.}
We write $u = \sum_{i\geq 1} \omega_i$ with $\omega_i \in E_i$. Since $u$ is a Maurer-Cartan element, we have
$d\omega_1= 0$. 
From $H({\mathcal L}_X,d)= 0$, we deduce the existence of $\omega_1'$ such that 
$\omega_1 = d\omega_1'$ and  \lemref{lem:MCr} implies $u\sim_{O(1)} u_2$ with $u_2\in E_{\geq 2}$. 
With the same process, we get a sequence of Maurer-Cartan elements $u_n\in E_{\geq n}$ such that 
$u_n\sim_{O(n)} u_{n+1}$. 
Finally \lemref{lem:seqMC} gives  $u\sim 0$.
 \end{proof}
 
\providecommand{\bysame}{\leavevmode\hbox to3em{\hrulefill}\thinspace}
\providecommand{\MR}{\relax\ifhmode\unskip\space\fi MR }
\providecommand{\MRhref}[2]{%
  \href{http://www.ams.org/mathscinet-getitem?mr=#1}{#2}
}
\providecommand{\href}[2]{#2}

\end{document}